%% file: main.tex
\documentclass[11pt,reqno]{amsart}
\usepackage[english]{babel}
\usepackage[latin1]{inputenc}
\usepackage[T1]{fontenc}
\usepackage{amssymb,amsmath,amstext,amsfonts,amsthm}
\usepackage{mathtools}
\usepackage{verbatim}
\usepackage{relsize}
\usepackage{graphicx,color}
\usepackage[dvipsnames]{xcolor}
\usepackage[all]{xy}
\usepackage[colorlinks=true, urlcolor=blue, linkcolor=blue, citecolor=blue]{hyperref}
\usepackage{fancyvrb}
\usepackage{tikz-cd}
\usepackage{comment}
\usepackage{hhline}
\usepackage{epigraph}
\usepackage{thm-restate}

\newtheorem{theorem}{Theorem}[section]
\newtheorem{proposition}[theorem]{Proposition}

\newtheorem{example}[theorem]{Example}
\newtheorem{conjecture}{Conjecture}[section]
\newtheorem{cor}[theorem]{Corollary}
\newtheorem{lemma}[theorem]{Lemma}

\theoremstyle{definition}
\newtheorem{definition}[theorem]{Definition}

\newcommand{\bx}{\textbf{x}}

\newcommand{\CC}{{\mathbb C}}
\newcommand{\RR}{{\mathbb R}}

\newcommand{\ZZ}{{\mathbb Z}}
\newcommand{\PP}{{\mathbb P}}

\newcommand{\OO}{{\mathcal O}}
\newcommand{\SOS}{{\mathrm{SOS}_k}}
\newcommand{\Sym}{\operatorname{Sym}}

\newcommand{\rk}{\operatorname{rk}}

\newcommand{\im}{\operatorname{Im}}
\newcommand{\codim}{\operatorname{codim}}

\DeclareMathOperator{\Span}{span}
\makeindex
\DeclarePairedDelimiter{\norm}{\|}{\|}

\makeatletter
\@namedef{subjclassname@2020}{%
	\textup{2020} Mathematics Subject Classification}
\makeatother

\title{On the degree of varieties of sum of squares}
\author[Ferguson]{Andrew Ferguson}
\author[Ottaviani]{Giorgio Ottaviani}
\author[Safey El Din]{Mohab Safey El Din}
\author[Teixeira Turatti]{Ettore Teixeira Turatti}
\address{Sorbonne Universit\'e, LIP6, 4 place Jussieu, 75252, Paris, CEDEX 05,
France}\email{andrew.ferguson@lip6.fr}
\address{Universit\`{a} di Firenze, Dipartimento di Matematica e Informatica, Viale Morgagni 67/A, 50134 Firenze, Italy}\email{giorgio.ottaviani@unifi.it}
\address{Sorbonne Universit\'e, LIP6, 4 place Jussieu, 75252, Paris, CEDEX 05,
France}\email{mohab.safey@lip6.fr}
\address{Universit\`{a} di Firenze, Dipartimento di Matematica e Informatica, Viale Morgagni 67/A, 50134 Firenze, Italy}\email{ettore.teixeiraturatti@unifi.it}


\begin{document}
\maketitle

     \begin{abstract}
        We study the problem of how many different sum of squares decompositions a general polynomial $f$ with SOS-rank $k$ admits. We show that there is a link between the variety $\mathrm{SOS}_k(f)$ of all SOS-decompositions of $f$ and the orthogonal group $\mathrm{O}(k)$. We exploit this connection to obtain the dimension of $\mathrm{SOS}_k(f)$ and show that its degree is bounded from below by the degree of $\mathrm{O}(k)$. In particular, for $k=2$ we show that $\mathrm{SOS}_2(f)$ is isomorphic to $\mathrm{O}(2)$ and hence the degree bound becomes an equality. Moreover, we compute the dimension of the space of polynomials of SOS-rank $k$ and obtain the degree in the special case $k=2$.
    \end{abstract}

\input{introduction}
\input{prelim}
\input{soskf}

\input{two_squares}

\section*{Acknowledgements}

This work has been supported by European Union's Horizon 2020 research and innovation programme under the Marie Sk{\l}odowska-Curie Actions, grant agreement 813211 (POEMA).

\bibliographystyle{alpha}
\bibliography{bib}
\end{document}

%% file: introduction.tex
\section{Introduction}\label{sec:intro}
\paragraph{\bf Motivation}  Let $V$ be a complex vector space of dimension $n+1$ with basis $\{x_0 \dots, x_n\}$ and let $d\geq 0$ be an integer. Let $f \in \CC[x_0,\dots,x_n]$ be a homogeneous polynomial of degree $2d$, that is $f \in \Sym^{2d}V$. A starting case, when $f$ is real, is the problem of computing the global infimum of $f$, $f^* = \inf_{z \in \RR^n} f(z)$. Polynomial optimisation problems appear frequently in practice in many different fields, including areas of engineering and social science such as computer vision~\cite{probst2019convex,Aholt_2013}, control theory~\cite{henrion2005,henrion2003opt} and optimal design~\cite{castro2017approximate}. However, even for $\deg f\geq 4$ this is an NP-hard problem~\cite{Murty1987SomeNP}. As such, many methods
have been developed to approximate $f^*$. A popular method is to relax the optimisation problem: 
\begin{align*}
    \max_{\lambda \in \RR} \lambda \text{ s.t. } & f - \lambda = \sum_{i=1}^k g_i^2, \\
    & g_i \in \Sym^d V.
\end{align*}

Clearly, being a sum of squares implies non-negativity. It is well-known that these notions are equivalent {in two homogeneous variables}. However, due to the counter example by Motzkin this is not true in general~\cite{Motzkin}. 

In~\cite{lasserre0}, using the duality between moments and sums of squares, Lasserre constructed a hierarchy of semi-definite programs whose solutions converge to the true infimum $f^*$. However, in general, the decompositions obtained from semi-definite programming are \emph{approximate} certificates of non-negativity. In recent years there has been an increased study on computing \emph{exact} certificates~\cite{peyrl08,magron2018}. Hence, one wants to understand the algebraic structure of SOS decompositions and the related semi-definite programs.

\paragraph{\bf Prior works}
Following the classical works of Sylvester~\cite{sylvester51}, the study of so-called Waring decompositions, decompositions of homogeneous polynomials by powers of linear forms, is an active area of research.
{In \cite{FOS} it was proved that any general $f\in\Sym^{2d}V$ is a sum of at most $2^n$ squares. For fixed $n$, this bound is sharp for all sufficiently large $d$}. 
%In the special case of SOS decompositions, 
The authors of~\cite{lundqvist2019generic} investigate the minimal numbers of squares in a decomposition of a generic polynomial in two variables. Then, in~\cite{froberg2018algebraic}, the authors give a conjecture on the generic SOS-rank of polynomials, see Definition \ref{def:sosk}, in terms of number of variables and degree. On the other hand, in this paper we will study generic polynomials of a given SOS-rank.

%In~\cite{nie2010algebraic}, Nie, Ranestad and Sturmfels study the algebraic degree of semi-definite programming. Specifically, they compute the degree of a solution to a semi-definite program. By duality, this result is then translated into the language of SOS decompositions. However, the formulae they provide requires that the parameters of the program, which in terms of the SOS formulation are the number of variables, $n$, the SOS-rank, $k$, and the degree of the polynomial, $d$, satisfy the following equation: 
%\begin{equation}\label{eq: NRS}
%       \binom{n+2d}{2d} = \binom{\binom{n+d}{d} - k + 1}{2}.
%\end{equation}
%We note that this equality is rarely satisfied outside of a few special cases.
%\red{satisfy a certain formula. Giorgio: to be clarified} 

%In this paper, one aim is to extend this result by analysing the SOS decompositions directly from an algebraic geometry point of view. Another aim is to better understand the structure of the SOS decompositions of a given polynomial.

In this paper, one aim is to analyse the degree of SOS decompositions directly from an algebraic geometry point of view. Another aim is to better understand the structure of the SOS decompositions of a given polynomial.

\paragraph{\bf Main results} We consider SOS decompositions of polynomials of degree $2d$. 
    \begin{definition}\label{def: genrank}
	    Let $f\in\Sym^{2d}V$. The polynomial $f$ has SOS-rank $k$ if $k$ is the minimum number such that there exist $f_i \in \Sym^{d}V$ such that
	    $$
	    f=\sum_{i=1}^r f_i^2.
	    $$
	    
	\end{definition}
	In this paper, we define and study two varieties related to exact SOS decompositions. The first is defined by all polynomials of rank less than or equal to $k$, with a general point $f\in \mathrm{SOS}_k$ being a polynomial of rank $k$.
	\begin{definition}\label{def:sosk}
		Let $\mathrm{SOS}_k$ be the subvariety in $\Sym^{2d}V$ obtained from the Zariski closure of the set of all SOS-rank $k$ polynomials.
		$$
		\mathrm{SOS}_k=\overline{\{f_1^2+\dots+f_k^2\ |\ f_i\in \Sym^dV  \}}.
		$$
		
		The generic SOS-rank is the smallest number $k$ such that $\mathrm{SOS}_k$ covers the ambient space.
	\end{definition}
	
	Another notion that can be explored is instead of analysing all polynomials of a given rank, one can seek to understand all the different decompositions of the general polynomial $f$.
	
	\begin{definition}\label{def:soskf}
	    Let $f\in \mathrm{SOS}_k$ be a generic polynomial. We define the variety of all the SOS decompositions of $f$ as $$
	    \mathrm{SOS}_k(f)=\{(f_1,\dots,f_k)\in{\prod_{i=1}^k \Sym^dV|\sum_{i=1}^kf_i^2=f}\}.
	    $$
	\end{definition}
	While we investigate the $\SOS(f)$ variety for all ranks $k$, in particular we give a complete description of the $k=2$ case.
\begin{restatable}{theorem}{sostwof}\label{thm:sos2f}
    Let $f\in\mathrm{SOS}_2\subset\mathrm{Sym}^{2d}V$, $\dim V=n+1>2$, be a generic polynomial that is the sum of two squares. Then, $\mathrm{SOS}_2(f)$ has two irreducible components isomorphic to $\mathrm{SO}(2)$.
\end{restatable}

{Since $\mathrm{SO}(k)$ acts on any decomposition using $k$ squares, we have the inequality $$\dim\mathrm{SOS}_k(f)\ge\dim\mathrm{SO}(k)=\binom{k}{2}.$$
In Corollary \ref{cor:okcomponent} we prove a statement which implies the following result.}

\begin{restatable}{theorem}{dimsoskf}\label{thm:dimsoskf}
    {Let $f\in\mathrm{SOS}_k$ be generic with $k \le n$. Then,
    $$\dim\mathrm{SOS}_k(f)=\binom{k}{2}.$$}
\end{restatable}

 By analysing the general polynomial in $\mathrm{SOS}_2$, we prove a formula for the degree of this variety.

\begin{restatable}{theorem}{degSOSk}\label{thm:degSOSk}
    Let $N=\dim \Sym^dV=\binom{n+d}{d}$. The degrees of the varieties of squares and of sum of two squares in $\PP(\Sym^{2d}V)$ are given by $$\deg(\mathrm{SOS}_1)=2^{N-1},\ \deg(\mathrm{SOS}_2)=\prod_{i=0}^{N-3}\frac{\binom{N+i}{N-2-i}}{\binom{2i+1}{i}}.$$
\end{restatable}

{Moreover, the dominant map
$$\begin{matrix}\pi\colon&\prod_{i=1}^k \Sym^dV&\to&\mathrm{SOS}_k\\
&(f_1,\dots,f_k)&\mapsto &\sum_{i=1}^kf_i^2\end{matrix}$$
has fibers $\pi^{-1}(f)=\mathrm{SOS}_k(f)$, so that 
Theorem \ref{thm:dimsoskf} implies the following.
\begin{cor}\label{cor:dimsosk}
$$\dim \mathrm{SOS}_k\le k\binom{n+d}{n}-\binom{k}{2}$$
and equality holds for $k\le n$ and a general $f\in\mathrm{SOS}_k$.
\end{cor}
}
\paragraph{\bf Structure of the paper} In Section~\ref{sec:pre}, we begin by recalling some definitions in sums of squares decompositions, algebraic geometry and commutative algebra. Then, in Section~\ref{sec:sos} we investigate the variety of all possible sums of $k$-squares decompositions of a given polynomial. We describe the action of the orthogonal group of size $k$ on this variety and conjecture that there is an isomorphism between these two objects. We provide experimental and theoretical support for this conjecture and conclude by showing that it holds for $k=2$. Finally, in Section~\ref{sec:two} we use the results of Section~\ref{sec:sos} to prove a formula for the degree of the variety of all SOS decompositions of two squares in addition to a upper bound on this degree for $k \geq 3$. %and relate these results to the work of Nie, Ranestad and Sturmfels~\cite{nie2010algebraic}.

%% file: prelim.tex
\section{Preliminaries}\label{sec:pre}
	Let $V$ be a complex vector space of dimension $n+1$. We will denote the $n$-dimensional projective space associated to $V$ by $\PP V$.
	
	\begin{definition}
		We define the $d$-Veronese embedding as the map $$
		\nu_d: \PP V\rightarrow \PP\Sym^d(V),\ \ell\mapsto \ell^d.
		$$
		
		Notice that the map $\nu_d$ is closed \cite{shafarevich2013basic}.	Therefore, we define the $d$-Veronese variety in $\PP \Sym^d(V)$ as the the image of $\PP V$ under the Veronese embedding $\nu_d$.
	\end{definition}
    \begin{definition}
        A polynomial $f\in\Sym^dV$ has rank one, or is decomposable, if $f=v^d$. The rank of a polynomial $f$ is defined as the minimum number $r\in\mathbb N$ such that $$
        f=\sum_{i=1}^r{v_i^d}.
        $$
        In other words, $f$ is the sum of $r$ decomposable polynomials.
    \end{definition}
    Observe that the Veronese variety $\nu_d(V)\subset \PP\Sym^d V$ consists exactly of the rank one polynomials.
	\begin{definition}
		Let $X$ be a subvariety of $V$. The $k$-th secant variety of $X$, denoted $\Sigma_k(X)$, is defined as the Zariski closure of the union of all the $k$ linear subspaces spanned by points in $X$. That is
		$$
		\Sigma_k(X)=\overline{\bigcup_{x_1,\dots,x_k\in X} \Span\{x_1,\dots,x_k\} }.
		$$
	\end{definition}
	If $X=\nu_d(\PP V)\subset \PP\Sym^dV$, then the generic elements in the $k$-th secant variety of the Veronese variety consist exactly of polynomials of rank $k$ as long the inclusion $\Sigma_k(\nu_d(\PP V))\subset \PP\Sym^d V$ is strict. 
	
	Let $U$ denote $\Sym^dV$. We can decompose $\Sym^{2d}U$ follows: $$
	\Sym^2U=\Sym^{2d}V\oplus C,
	$$
	where $C$ is {obtained by} plethysm, see \cite{weyman_2003} for more details. The space $C$ corresponds to the quadrics on $U$ that vanish on $\nu_d(\PP V)$. Moreover, $\Sym^{2d}V$ is {the degree two piece of the coordinate ring 
	of $\nu_d(\PP V)$}. 
	
	Let $\{x_0,\dots, x_n\}$ be a basis of $V$. Consider a basis $w_1=x_0^d,w_2=x_0^{d-1}x_1,\dots,{w_{N}}=x_n^d$, with $N=\binom{n+d}{d}$. A rank one quadric $q$ in $\Sym^2U$ has an expression $q=(\alpha_1w_1+\dots+{\alpha_{N}w_{N}})^2$, with  $\alpha_1,\dots,{\alpha_{N}}\in\CC$. Switching to the coordinates given by $V$ we have $$
		q=(\alpha_1x_0^d+\dots+{\alpha_{N}}x_n^d)^2.
	$$
	This means that rank one quadrics in $\Sym^2U$ correspond to square powers in $\Sym^{2d}V$. Furthermore, applying the same argument for a rank $k$ quadric $f\in\Sym^2U$, we see that $f$ corresponds to a sum of $k$ squares in $\Sym^{2d}V$.
	
	Notice that if $(f_1,\dots,f_k)\in \mathrm{SOS}_k(f)$, as defined in Definition~\ref{def:soskf}, then for any permutation $\sigma\in S_k$, where $S_k$ is the symmetric group of order $k$, we have that $$ \left(f_{\sigma(1)},\dots,f_{\sigma(k)}\right)\in \mathrm{SOS}_k(f).$$ One could desire to remove such "overlapping" points by taking the quotient by $S_k$. However, there is another important group, containing such permutations, that acts on $\SOS(f)$.
    
    Let $\mathrm O(k)$ be the orthogonal group of order $k$.
    %\red{Giorgio: $N$ was already defined, I have canceled its second definition.}
    %and let $N=\dim \Sym^{d}V=\binom{n+d}{n}$.
    Fix a point $(f_1,\dots,f_k)\in\mathrm{SOS_k}(f)$ and fix the ordering of the basis $\{w_1, \dots,{w_{N}}\}$ of {$\Sym^{d}V$}. Define $A$ to be the $k\times N$ matrix whose $i$-th row is the coefficients of the polynomial $f_i$. Then, $$
    xA^tAx^t=f.
    $$
    Let $O\in\mathrm{O}(k)$, then the action on the left by $A$ preserves the polynomial $f$. That is,
    $$
    x(OA)^tOAx^t=f.
    $$
    Essentially, such an action leads to a different decomposition $(f_1',\dots,f_k')$ of $f$, where $f_i'$ is the ith row of the matrix $OA$.
    
    \begin{definition}
    Let $f\in\Sym^dV$, let $\{x_0,\dots,x_n\}$ be a basis of $V$
    {and let $\partial_0,\ldots, \partial_n$ be the dual basis of $V^\vee$}. For each $m<d$, we define the linear map
    \begin{align*}
        W_f^m: & {\Sym^mV^\vee}  \to  \Sym^{d-m}V,\\
         &{\partial_{i_1}\cdots \partial_{i_m }}\mapsto \frac{\partial f}{\partial x_{i_1}\cdots \partial x_{i_m}}.
    \end{align*}
The matrix corresponding to this linear map is called the \emph{catalecticant} matrix of $f$.
    \end{definition}
    
    \medskip
    We give some cohomological definitions that are going to be used later on. Let $S=\bigoplus_q\Sym^q(V)$ be the symmetric algebra of $V$.
    
    \begin{definition}\label{def: Koszul complex vector space}
Let $R$ be a ring and $F$ a free module of rank $r$ over $R$. Given an $R$-linear map $k:F\to R$, the complex \begin{align*}
    0\to \bigwedge^rF\xrightarrow{{\varphi_{r}}}\bigwedge^{r-1}F\xrightarrow{{\varphi_{r-1}}}\dots\xrightarrow{\varphi_2}F\xrightarrow{\varphi_1}R\to0 
\end{align*}
is called the Koszul complex associated to $k$. The maps $\varphi_l$ are defined as $$
\varphi_l(e_1\wedge\dots\wedge {e_\ell})=\sum_{i=1}^\ell(-1)^{i+1}k(e_i)e_1\wedge\dots\wedge \widehat{e_i}\wedge\dots\wedge {e_\ell},
$$
where the notation $\widehat{e_i}$ means that this element is omitted from the product.
\end{definition}

    \begin{definition}\label{def: Betti}
     Let $M$ be a finitely generated graded $S$-module and let $F_0,\dots,F_m$ be the free $S$-modules that give a minimal free resolution of $M$. That is, there is an exact sequence $$
     0\to F_m\to F_{m-1}\to\dots\to F_1\to F_0\to M\to 0,
     $$
     and the matrices of the maps $\phi_i:F_{i+1}\to F_i$ have no non-zero constant entry, see \cite{eisenbud1995commutative}. The Betti number $\beta_{i,j}$ is the number of generators of degree $j$ needed to describe $F_i$. That is, $F_i=\bigoplus_j S(-j)^{\beta_{i,j}}$, where $S(-j)$ is the $j$-graded part of $S$.
\end{definition}

\begin{definition}\label{def: Tor}
Let $M$, $N$ be two graded $S$-modules and let $F_\bullet$ be a free resolution of $N$. Consider the complex $F_\bullet\otimes M$. The Tor groups are defined by $$
\mathrm{Tor}^S_p(M,N)=H^p(F_\bullet \otimes M).
$$
\end{definition}

The next result shows the relation between the $\mathrm{Tor}$ groups of $M$ and the Betti numbers of a free resolution of $M$.

\begin{proposition}\cite[Section 1]{Green}\label{prop: Green}
Let $\mathfrak{m}\subset S$ be the maximal ideal $\mathfrak{m}=\bigoplus_{q\geq1} \Sym^q(V)$ and let $\underline{k}=S/\mathfrak{m}$ be the residual field. Then,
$\mathrm{Tor}^S_p(M,\underline{k})_q$ has rank equal to $\beta_{p,q}$.
\end{proposition}

This connection between the Betti numbers and the Tor groups is important because it correlates the Betti numbers with cohomology. This allows us to use semi-continuity on the Betti numbers, as explained in the next theorem.

\begin{theorem}\cite[Theorem 12.8]{Hart}\label{thm: semicont}
Let $f:X\rightarrow Y$ be a projective morphism of noetherian schemes. Let $\mathcal F$ be a coherent sheaf on $X$ and flat over $Y$, in other words, $\mathcal F$ is a finitely presented $\OO_X$-module and the functor $\text{\textendash}\otimes \OO_{Y,f(x)}:\mathrm{Mod}_{\mathcal F_x}\to\mathrm{Mod}_{\mathcal F_x}$ is exact for every $x\in X$. Then for each $i\geq0$, the function 
$$
y\mapsto \dim H^i(X_y,\mathcal F_y)
$$
is upper semi-continuous on $Y$.
\end{theorem}

%% file: soskf.tex
\section{The degree of the variety of all SOS decompositions}\label{sec:sos}
Let $f = \sum_{i=1}^k f_i^2 \in \CC[x_0, \dots, x_n]$ be a sum of squares with degree $2d$. We consider the variety  in the ambient space $\prod_{i=1}^k Sym^d(V)$ of all possible SOS decompositions of the given polynomial $f$. 
$$\text{SOS}_k(f) = \{(f_1,\ldots, f_k)\in \prod_{i=1}^k Sym^d(V)|\sum_{i=1}^k f_i^2=f\}$$
We conjecture the degree of this variety, when $n \geq k$, to be the degree of the orthogonal group $\text{O}(k)$. In \cite{brandt2017degree} the authors give the degree of $\text{SO}(k)$, and thus $\text{O}(k)$, to be the determinant of the following binomial matrix

\begin{equation}\label{eq:degort}
    \deg \text{O}(k) = 2^{k} \det \left ( \left( \binom{2k - 2i - 2j}{k-2i} \right)_{1 \leq i,j \leq \lfloor{}\frac{k}{2} \rfloor{}} \right).
\end{equation}
For the case $d=1$, the {argument} simplifies and so we give the following lemma.
\begin{lemma}
Let $f \in \Sym^2V$ be a quadric of $\mathrm{SOS}$-rank $k \leq n$. Then, in the affine setting, the degree of $\mathrm{SOS}_k(f)$ is equal to the degree of $\mathrm{O}(k)$.
\end{lemma}

\begin{proof}
With $f= \sum_{i=1}^k f_i^2$, $n \geq k$ implies that we can encode $f$ in a $k \times (n+1)$ matrix, $A$, whose rows give the coefficients of the linear forms $f_i$. Then, with $\bx = (x_0, \dots, x_n)$ we have that $\norm{Ax^t}^2 = f$. Thus, for any orthogonal matrix $O\in\mathrm{O}(k)$ we have that \[ \norm{OA\bx^t}^2 = (OA\bx^t)^t (OA\bx^t) = \bx A^tO^t OA\bx^t = \bx A^tA\bx^t = \norm{A\bx^t}^2 = f\] Hence, there is an action on the $\mathrm{SOS}_k(f)$ variety by $\mathrm{O}(k)$. Additionally, there are at least two identical irreducible components that correspond to $\det O = \pm 1$.

We now show that up to a change of coordinates and multiplication by an orthogonal matrix, this SOS expression is unique. Let $A$ and $B$ be $k \times (n+1)$ matrices encoding SOS decomposition of $f$. Then, up to a change of coordinates, we can ensure that the first $k$ columns are linearly independent and so $\mathrm{QR}$ decompositions can be found. Thus, let $A = Q_1R_1$ and $B = Q_2R_2$ where $Q_1,Q_2$ are $k \times k$ orthogonal matrices and $R_1, R_2$ are $k \times (n+1)$ upper triangular matrices. Then, $R_1$ and  $R_2$ also encode SOS decompositions of $f$.  By the equation $\norm{R_1x^t}^2 = f$ we can identify exactly the entries of $R_1$, up to multiplication by $\pm 1$ in the rows, or in other words, up to multiplication by an orthogonal matrix. The same holds for $R_2$ and so the decompositions encoded by $A$ and $B$ must be in the same orbit of the action of $\text{O}(k)$ on $\text{SOS}_k(f)$. Therefore, there is only one orbit and so the degree of $\text{SOS}_k(f)$ is equal to the degree of $\text{O}(k)$. 
%Moreover, if we restrict to $\text{SO(k)}$ then we reduce to an irreducible variety with half the degree.
\end{proof}

The argument above also works partially for the case $d \geq 2$. Once a basis is chosen for $\Sym^d V$, we can construct the matrix in the same way with $k$ rows but $\binom{n+d}{d}$ columns. Then, the group $\mathrm{O}(k)$ acts on the left to give new decompositions. However, the QR decomposition no longer implies uniqueness of the orbit. This is because there exist relations between the monomials described by the columns of $f$. In other words, the Gram matrix associated to $f$ is not only symmetric but also has a moment structure. Thus, it is no longer easy to see that the non-linear equations given by the norm of $Ax^t$ squared, $\norm{Ax^t}^2$, have a unique solution. 

Experimentally, up to $k \leq 6$, we observe a stabilisation of the degree of the variety $\SOS(f)$ as the degree of $f$ increases. The following table derives from \cite[Table~1]{brandt2017degree}. Since the degree of $\mathrm{SOS}_7(f)$ is at least $233,232$ for a generic $f \in \mathrm{SOS}_7$, $k \leq 6$ is the currently the limit for our experimental methods. 
    
\begin{table}[htb]
    \centering
    \begin{tabular}{cccc}
         \hline
         $k$ & Symbolic & Formula ($\text{O}(k)$) & Formula ($\text{SO}(k)$) \\
         \hline
         $2$ & 4 & 4 & 2\\
         $3$ & 16 & 16 & 8\\
         $4$ & 80 & 80 & 40 \\
         $5$ & 768 & 768 & 384\\
         $6$ & 9356 & 9356 &4768\\
         $7$ & - & 233232 &111616\\
         $8$ & - & 6867200 & 3433600\\
         $9$ & - & 393936896 & 196968448 \\
         \hline
    \end{tabular}
    \caption{Degree of $\text{SOS}_k(f)$ for $n \geq k$. See formula~\ref{eq:degort} for the degree of $\text{O}(k)$.}
    \label{tab:Ok}
\end{table}
The next example shows that 
 %also for $n>1$ 
 the condition $n\geq k$ is sharp.
 
 \begin{example}
  The general plane quartic can be expressed as $g_1^2+g_2^2+g_3^2$ in $63$ ways, where $g_i\in\Sym^2\CC^3$. A proof of such result is presented in \cite[Theorem 6.2.3]{Dolg}. The idea is to consider the quartic form as the determinant of a $2\times2$ matrix whose entries are quadric forms.
 \end{example}

The next lemma gives an indication of the connection between $k$-SOS decompositions and the orthogonal group $\mathrm O(k)$. 
    %The next lemma gives an indication of why the SOS decomposition is connected with the orthogonal group $\mathrm O(k)$. 
    Indeed, fixing a matrix $A_0\in \mathcal M_{k\times N}$ is equivalent to fixing a sum of squares decomposition of rank $k$ of $f=x^tA_0^TA_0x$.
    
    \begin{lemma}\label{lem:Ok}
        Let $N\geq k\geq 1$ be integers and $A,A_0\in \mathcal M_{k\times N}$ be matrices, $A_0$ of maximal rank and consider the entries of $A$ as variables $x_{ij}$. Then the variety $Y$ defined by the equation $$
        A^tA=A_0^tA_0
        $$
        is isomorphic to $\mathrm O(k)$.
    \end{lemma}
    \begin{proof}
        Up to an action of the group of $N \times N$ invertible matrices, $\mathrm{GL}(N)$, on the left and $\mathrm O(k)$ on the right of $A_0$, we may suppose without loss of generality that $A_0=\begin{bmatrix}I_k&0
        \end{bmatrix}
        $, with $I_k$ the $k\times k$ identity matrix and $0$ a null matrix of size $k\times (N-k)$. Let $A=\begin{bmatrix}X_0&X_1 \end{bmatrix}$, again with $X_0$ a $k\times k$ matrix and $X_1$ a $k\times (N-k)$ matrix.
        
        In those coordinates, the variety is determined by $$
        \begin{bmatrix}
        X_0^tX_0&X_0^tX_1\\
        X_1^tX_0&X_1^tX_1
        \end{bmatrix}=\begin{bmatrix}
        I_k&0\\
        0&0
        \end{bmatrix}.
        $$
        
        The first block implies that $X_0^tX_0\in \mathrm O(k)$. Moreover $X_0^tX_1=0$ implies that $X_1=0$ since $X_0$ is invertible.
    \end{proof}

Let $f\in \Sym^{2d}V$ be a sum of $k$ squares $f=\sum_{i=1}^kf_i^2$. Then, $f=xA^tAx^t$, where $A\in\mathcal M_{k\times N}$ has the coefficients of $f_i$ as its $i$-th row. This gives a natural isomorphism \begin{equation}\label{eq:soskB}
\mathrm{SOS}_k(f)\cong \{B\in\mathcal M_{k\times N}|xB^tBx^t=f\}.
\end{equation}

Denote the Gram matrix $W_A=A^tA$ and note that $\rk W_A = k$ when the above decomposition is minimal. The previous lemma implies that $\mathrm O(k)\cong \{B\in\mathcal M_{k\times N}|W_B=W_A\}\subset \mathrm{SOS}_k(f).$ This, together with the isomorphism \eqref{eq:soskB} implies that $\mathrm{SOS}_k(f)$ can be described by as many copies of $\mathrm{O}(k)$ as the number of distinct symmetric matrices $W_A$ of rank $k$ such that $x^tW_Ax=f$.

Let $f\in \Sym^{2d}V$ and $N=\binom{n+d}{d}$. Notice that the following diagram commutes.
\begin{equation}\label{eq: diagram}
\begin{tikzcd}
\mathbb C^k\otimes\mathbb C^N \arrow[rr, "\varphi", "A\ \longmapsto\ A^tA"']\arrow[rrrr, bend left, "A\ \mapsto xA^tAx^t"] &   & \Sym^2(\mathbb C^N) \arrow[rr, "\pi", "B\ \longmapsto\ xBx^t"'] &  & \Sym^{2d} V
\end{tikzcd}
\end{equation}

We have that $\mathrm{SOS}_k(f)=\{A|\pi(A^tA)=f \}$. Moreover, if $B\in \im \varphi$ then $\rk B\leq k$.

The fiber $\pi^{-1}(f)= W_0+C$, where $W_0$ is the rank $N$ catalecticant matrix of $f$ such that $xWx^t=f$ and $C$ is the variety \[C = \{C_0 \in S^N \; | \; x^T C_0 x = 0\}.\] 

This means that the problem can be reformulated in terms of the intersection $\varphi(\CC^k\otimes \CC^{N})\cap (C+W_0)$: when this intersection is just a single point, as is the case for $k \leq 6$ shown in Table~\ref{tab:Ok}, this implies that there exists only one $C_0\in C$ such that $W_0+C_0$ has rank $k$. This is equivalent to saying that $\mathrm{SOS}_k(f)$ consists of a single copy of $\mathrm{O}(k)$. Thus, we arrive at the following conjecture.
\begin{conjecture}\label{conj}
Let $f \in \Sym^{2d}V$ be generic of SOS-rank $k \leq n$ and let $N = \binom{n+d}{d}$. Then, 
$$\mathrm{SOS}_k(f)\cong \{A \in \CC^k \otimes \CC^N \mid A^tA=W,\ xWx^t=f\}=\{A \in \CC^k \otimes \CC^N \mid A^tA=W_0+C_0\}=\mathrm O(k).$$
\end{conjecture}

Of course, if this intersection consists of more than a single point, one would arrive at exactly the number of copies of $\mathrm O(k)$ such that $\mathrm{SOS}_k(f)$ is isomorphic. 

Consider a tuple $(f_1,\dots,f_k)\in \mathrm{SOS}_k(f)$, we denote the tangent space of $\mathrm{SOS}_k(f)$ at this point by $\mathrm{TSOS}_k(f)_{(f_1,\dots,f_k)}$. Recall that if we consider an orthogonal matrix $O\in \mathrm O(k)$ and $A_f\in \mathcal M_{k\times N}$, then the rows of $A_fO$ are polynomials giving a $k$-SOS decomposition of $f$. 

We are interested in understanding the local behavior of this variety. More specifically, we want to show that the tangent space $\mathrm{TSOS}_k(f)_{(f_1,\dots,f_k)}$ has dimension equal to the dimension of $\mathrm O(k)$. This means that locally, the variety $\mathrm{SOS}_k(f)$ is exactly equal to $\mathrm O(k)$. In order to do that, we can show that the only syzygies of a vector $(f_1,\dots,f_k)\in \mathrm{SOS}_k(f)$ are given by the Koszul syzygies. In the next paragraphs, we further explain the concept of Koszul syzygies and how they are related to the tangent space of $\mathrm{SOS}_k(f)$.

Let $A_f$ be the matrix whose rows are the coefficients of $f_1,\dots,f_k$. Observe that the map $$\phi:A\mapsto xA^tAx^t-f$$ gives $\mathrm{SOS}_k(f)$ as the fiber at zero. Therefore, the tangent space $\mathrm{TSOS}_k(f)_{(f_1,\dots,f_k)}$ 
%$\mathrm{TSOS}_k(f)_{A_f}$ 
is the space generated by the nullity of the derivative of $\phi$ at the point $(f_1,\dots,f_k)$. This equivalent to saying that 
\begin{equation}\label{eq:skew}
x(A_f^tV+V^tA_f)x^t=0    
\end{equation}
 where $V\in\mathcal M_{k\times N}$. Notice that equation~\eqref{eq:skew} is trivially satisfied when $A_f^tV$ is a skew-symmetric matrix. A syzygy satisfying this equation is a Koszul syzygy of $(f_1,\dots,f_k)$. If we have that the Koszul syzygies are the only syzygies of the point $(f_1,\dots,f_k)$, we obtain that they span the tangent space at this point. In such case, the tangent space has dimension equal to the dimension of $\mathrm{O}(k)$.

A more geometric and intuitive explanation can be described by looking at the usual set of coordinates instead of matrices. We may see $\mathrm{SOS}_k(f)$ as the nullity of the map $$\varphi:(h_1,\dots,h_k)\mapsto \sum_{i=1}^k h_i^2 -f.$$ The tangent space $\mathrm{TSOS}_k(f)_{(f_1,\dots,f_k)}$ is computed once again as the space generated by the nullity of the derivative of the expression $\sum_{i=1}^kh_i^2-f$ at the point $(f_1,\dots,f_k)$, that is $\varphi'(f_1,\dots,f_k)=0$. This means that the tangent space is generated by $$\sum_{i=1}^kf_ig_i=0.$$ The vanishing of this expression by considering tuples $(g_1,\dots,g_k)$ such that we have pairs $i\neq j$ with $g_i=f_j$ and $g_j=-f_i$ is a Koszul syzygy of the vector $(f_1,\dots,f_k)$. Observe that this corresponds exactly to the matrix $A_f^tV$ being skew-symmetric, where $V$ is the matrix that has $g_i$ as the ith-row.

\begin{proposition}
    The only syzygies of the vector $(x_0^{d},\dots,x_k^{d})$ are the Koszul syzygies.
\end{proposition}
\begin{proof}
    Let $A=\begin{bmatrix}
    I&0
    \end{bmatrix}$
    be a matrix as in equation (\ref{eq:soskB}) giving a SOS decomposition of $f=x_0^{2d}+\dots+x_k^{2d}$ in a basis $\{x_0^d,\dots,x_n^d,\dots\}$, where $I$ is the $k\times k$-identity matrix. Consider $V=\begin{bmatrix}
    v_{ij}
    \end{bmatrix}$ a $k\times N$-matrix, then $\partial \varphi(A)=V^tA+A^tV$ is the derivative of $\varphi$. The statement is equivalent to show that $x\partial\varphi(A)x^t=0$ if and only if $V^tA$ is skew-symmetric.% since this condition means the vanish by the Koszul syzygy. 
    
    In this basis, $W=V^tA+A^tV=\begin{bmatrix}
    v_{ij}+v_{ji}
    \end{bmatrix}$, and $$xWx^t=\sum_{i=1}^k\bigg(\sum_{j=1}^k(v_{ij}+v_{ji})x_j^d\bigg)x_i^d=0,$$
    since each monomial coefficient is equal to zero we obtain $2(v_{ij}+v_{ji})=0$ as desired.
\end{proof}

The importance of this result is that it guarantees that at the point $(x_0^d,\dots,x_k^d)$ the tangent space to $\mathrm{SOS}_k(x_0^{2d}+\dots+x_k^{2d})$ has dimension equal to the number of Koszul syzygies, since they span the null space of $\varphi'(x_0^d,\dots,x_k^d)$. Moreover, this dimension is exactly equal to the dimension of the tangent space of $\mathrm O(k)$. This implies that locally at the point $(x_0^d,\dots,x_k^d)$, the variety $\mathrm{SOS}_k(f)$ is equal to the subvariety $\mathrm O(k)\subset \mathrm{SOS}_k(f)$. We wish to extend this result to every point $(f_1,\dots,f_k)\in \mathrm{SOS}_k(f)$. We obtain that this can be extended to a vector $(f_1,\dots,f_k)$ by means of semi-continuity.
Indeed, consider the kernel $K$ of the map
$$\OO_{\PP V}(-d)^k\xrightarrow{[f_1\ldots f_k]}\OO_{\PP V}$$ defined by the vector $(f_1,\ldots, f_k)$, where $\OO_{\PP V}$ is the sheaf defining $\PP V$ as a scheme $(\PP V,\OO_{\PP_V})$.
The minimal resolution of the kernel, when there are only Koszul syzygies, start with
$$\ldots\to\OO_{\PP V}(-2d)^{\binom{k}{2}}\to K\to 0$$
By Proposition~\ref{prop: Green}, the Betti numbers $\beta_{p,p+q}$ of the minimal resolution of $K$ correspond
to the rank of $\mathrm{Tor}_p^S(K,\underline k)_{p+q}$, this is the component of degree $p+q$ of $Tor_p^S(K,\underline k)$.  %$\cong H^p(E_\bullet\otimes K)$, where $S=\bigoplus_{d\geq0} \Sym^d \CC^n$, $\underline k$ is the residual field of the maximal ideal $\bigoplus_{d\geq1} \Sym^d \CC^n$ and $E_\bullet$ the Koszul complex of $\underline k$. 
Since we can correlate the Betti numbers with cohomology dimensions using Proposition~\ref{prop: Green}, we have by Theorem~\ref{thm: semicont} that for a local deformation of $K$, the Betti numbers satisfy semi-continuity. Moreover, since we know that for any other point $(f_1,\dots,f_k)$ will have at least the Koszul syzygies, this implies that it will have only them.
\begin{cor}
    Suppose that $k\le n$ and $f\in\mathrm{SOS}_k$ is general. Let $(f_1,\dots,f_k)$ be a vector in $(\Sym^d V)^{\times k}$ giving the decomposition as $k$ sum of squares of a polynomial $f$. Then the only syzygies of $(f_1,\dots,f_k)$ are the Koszul ones.
\end{cor}
\begin{cor}\label{cor:okcomponent}
    Suppose that $k\le n$ and $f\in\mathrm{SOS}_k$ is general. We have an isomorphism $\mathrm{SOS}_k(f) \cong \mathrm O(k)^p$, for some $p \in \ZZ_+$. Note that this does not depend on the degree of $f$. In particular $\deg \mathrm{SOS}_k(f)\ge\deg \mathrm O(k)$ which is computed in Table~\ref{tab:Ok}, in fact $\deg \mathrm{SOS}_k(f) \equiv 0 \mod \deg O(k)$.
\end{cor}

We notice that the diagram \ref{eq: diagram} can have its conclusion interpreted in a different manner. Instead of considering $W_0$ a maximal rank matrix, one may consider a fixed matrix $A_0$ defining $f$, and let $$\mathrm{SOS}_k(f)=\{B^TB+C_0\ |\ \mathrm{rank}(B^TB+C_0)=k,\ B^TB=A_0^TA_0 \text{ and } C_0\in C\}.$$ Notice that such interpretation means that adding $C_0\neq 0$ is equivalent to changing the $\mathrm O(k)$ component of $\mathrm{SOS}_k(f)$. Thus, if there exists no other matrix $C_0$ besides $0$ such that $\mathrm{rank}(A_0^TA_0+C_0)=k$, it implies that there exists only one component. 

In the next pages we explore this equivalent problem and compare the dimensions of symmetric matrices of rank $k$ and $C$. Although a proof that the only translation by $C$ preserving the rank is $0$ is not obtained, by a comparison of dimensions we get a clear indicator that we should not expect other solutions. 

Let $S_k^N$ be the variety of symmetric matrices of size $N = \binom{n+d}{d}$ of rank at most $k$. Then, for some fixed $W \in S_k^N$, consider the variety \[(S+W)^N_k = \{B \; | \; B+W \in S_k^N\}.\] Note that this is indeed a variety as it is defined by the minors of the matrix $B+W$ and moreover, for all $M \in (S+W)^N_k$, we have that $M-W \in S_k^N$. Hence, we can consider this variety a translation of $S_k^N$ by the matrix $W$. \[(S+W)^N_k = S_k^N - W.\] Recall the variety $C = \{C_0 \in S^N \; | \; x^T C_0 x = 0\}$ and note that the following statement holds: \[ \text{For a generic } W, \; (S+W)_k^N \cap C = 0 \iff \text{For a generic } f, \; \deg \mathrm{SOS}_k(f) = \deg\mathrm O(k). \] 

Firstly, note that since $W$ is symmetric of rank $k$, there exists a decomposition of the form $W = A^T A$ where $A \in \mathcal{M}_{k \times N}$. Then, since every symmetric matrix of size $N$ gives a polynomial, through a moment vector $x$, we obtain a decomposition of $x^T W x$ as a sum of $k$ squares as $x^T A^T A x$. Then, as is discussed above, we would obtain equality for Corollary~\ref{cor:okcomponent}.

From the translation argument above, we obtain the following equivalences, \[(S+W)_k^N \cap C = 0 \iff (S_k^N - W) \cap C = 0 \iff S_k^N \cap (C+W) = W.\]

The equations defining $C$ are not general. Each equation specifies that a particular coefficient in the expansion of $x^tBx$ be zero. Hence, no coefficients of a general $f$ are zero, we have that a generic $W$ is not contained in the hyperplanes defined by any of the $\binom{n+2d}{2d}$ equations defining $C$.

Let $N = \binom{n+d}{d}$ and let $S$ be the polynomial ring $\CC[x_{ij}|1 \leq i,j \leq N]$. We set $x_{ij} = x_{ji}$ and consider $X = (x_{ij})_{1\leq i,j \leq N}$ to be an $N\times N$ variable symmetric matrix. For $1 \leq k \leq N-1$, we denote by $I_k$ the ideal generated by the $k+1$ minors of $X$. It is known that $S/I_k$ is a Cohen-Macaulay normal domain with dimension \[ \dim S/I_k = \frac{(2N + 1 - k)k}{2}. \]

Then, recall that \[\Sym^2(\Sym^d V)=\Sym^{2d}V\oplus C.\] Thus, \[ \codim C = \dim  \Sym^{2d} V = \binom{n+2d}{2d}.\]
The following lemma, through a dimension count, gives further support for Conjecture~\ref{conj}.
\begin{lemma}
    Let $k \leq n$. Then, for all $n,d \geq 1$, $\dim S/I_k < \codim C$.
\end{lemma}

\begin{proof}
Firstly, note that $\dim S/I_k$ is maximal when $k = N = \binom{n+d}{d}$ and that the dimension decreases monotonically as $k$ decreases. However, since we restrict to $k \leq n$, it suffices to show that $\dim S/I_n < \codim C$.
Now, suppose that $d=1$. Then,
\begin{align*}
    \dim S/I_n - \codim C = & \frac{(2(n+1) + 1 - n)n}{2} - \frac{(n+2)(n+1)}{2} \\
    = & \frac{n(n+3) - (n+1)(n+2)}{2} \\
    = & -1.
\end{align*}
Next, consider $d \geq 2$. Note that for all $n \geq 1$, \[ \frac{(2\binom{n+d}{d} + 1 - n)n}{2} \leq n\binom{n+d}{d}. \] 
Hence, it suffices to prove that for all $d \geq 2$, \[ \binom{n+2d}{2d} > n\binom{n+d}{d}. \]
We proceed by induction on $d$. In the base case $d=2$ we have, 
\begin{align*}
    \binom{n+4}{4} - n\binom{n+2}{2} & = \frac{(n+4)(n+3)(n+2)(n+1)}{4!} - \frac{(n+2)(n+1)n}{2!} \\
    & = \frac{(n+1)(n+2)}{4!}(n^2 - 5n + 12).
\end{align*}
It is easy to see that the polynomial $n^2 -5n +12$ is positive for all $n$ and so the base case holds. Now, assume for some fixed $d \geq 2$ that $\binom{n+2d}{2d} > n\binom{n+d}{d}$ and consider
\begin{align*}
    \binom{n+2d+2}{2d+2} - n\binom{n+d+1}{d+1} & = \frac{(n+2d+2)(n+2d+1)}{(2d+2)(2d+1)} \binom{n+2d}{2d} - \frac{n+d+1}{d+1} n \binom{n+d}{d} \\
    & = n\binom{n+d}{d} \left( \frac{(n+2d+2)(n+2d+1)}{(2d+2)(2d+1)} \frac{\binom{n+2d}{2d}}{n\binom{n+d}{d}} -  \frac{n+d+1}{d+1} \right) \\
    & > n\binom{n+d}{d} \left(\frac{(n+2d+2)(n+2d+1)}{(2d+2)(2d+1)} -  \frac{n+d+1}{d+1} \right) \\
    & = n\binom{n+d}{d} \left(\left(1+ \frac{n}{2d+2}\right)\left(1+ \frac{n}{2d+1}\right) - \left(1+ \frac{n}{d+1}\right) \right) \\ 
    & > n\binom{n+d}{d} \left(\left(1+ \frac{n}{2d+2}\right)^2 - \left(1+ \frac{n}{d+1}\right) \right) \\ 
& > n\binom{n+d}{d} \left(\left(1+ \frac{2n}{2d+2}\right) - \left(1+ \frac{n}{d+1}\right) \right) = 0. 
\end{align*} 
Thus, by induction, $\codim C - \dim S/I_k > 0$. 
\end{proof}

 We finish this section by proving that Conjecture~\ref{conj} holds for $k=2$.
 
\sostwof*

\begin{proof}
    Let $f\in \Sym^{2d}V$ be a general polynomial such that $f=g^2+h^2=(g+ih)(g-ih)$. Since $n>2$, then $f$ general is smooth and by consequence irreducible, thus the factorization is unique (UFD), we have for any other $g',\ h'$ such that $f=g'^2+h'^2$, then $\lambda(g'+ih')=g+ih$ and $\lambda^{-1}(g'-ih')=g-ih$, or $\lambda(g'+ih')=g-ih$ and $\lambda^{-1}(g'-ih')=g+ih$.
    
    Consider the first set of conditions, then $g=g'\frac{\lambda+\lambda^{-1}}{2}+h'\frac{i(\lambda-\lambda^{-1})}{2}$ and $h=g'\frac{\lambda-\lambda^{-1}}{2i}+h'\frac{\lambda+\lambda^{-1}}{2}$. Thus $$
    \begin{bmatrix}
    g\\h
    \end{bmatrix}=
    \underbrace{\begin{bmatrix}
    \frac{\lambda+\lambda^{-1}}{2} & \frac{i(\lambda-\lambda^{-1})}{2}\\
    \frac{\lambda-\lambda^{-1}}{2i} &\frac{\lambda+\lambda^{-1}}{2}
    \end{bmatrix}}_{A}\begin{bmatrix}
    g'\\h'
    \end{bmatrix}.
    $$
    Since $\det(A)=1$ and $AA^t=I$, this corresponds to one component of $\mathrm O(2)$. Then, the last copy of $\mathrm{SO}(2)$ is obtained from the other two conditions.
\end{proof}

%% file: two_squares.tex
\section{The degree of the variety of the sum of two squares}\label{sec:two}
	Let $V$ be a complex vector space of dimension $n+1$ and let $d\geq 0$ be an integer. Let $U=\Sym^dV$ and $\pi_C$ be the projection of $\Sym^2U=\Sym^{2d}V\oplus C$ centered at $C$, that is $$
	\pi_C:\Sym^{2d}V\oplus C\rightarrow \Sym^{2d}V.
	$$
	Notice that whenever $\Sigma_k(\nu_2(U))\cap C=0$, $\pi_C|_{\Sigma_k(\nu_2(U))}$ is a {well-defined morphism}. In such a case, {assuming $\pi_C|_{\Sigma_k(\nu_2(U))}$ is an isomorphism,} this means that $$\deg(\mathrm{SOS}_k)=\deg \big(\Sigma_{k}(\nu_2(U))\big).$$
	
	\begin{theorem}\label{intersection}
	Following the previous notation we have that $$\Sigma_{1}(\nu_2(\PP U))\cap   C=\emptyset \quad \text{ and } \quad \Sigma_{2}(\nu_2(\PP U))\cap   C=\emptyset.$$
	\end{theorem}
	\begin{proof}
	     It is known from the {Borel}-Weil Theorem, see \cite{weyman_2003}, that whenever $X=G/P\subset \PP(V_\lambda)$, where $G$ is an algebraic group and $P\subset G$ a parabolic group, then $H^0(X,\PP(V_\lambda(k)))=V_{k\lambda}$~\cite{serre1954representations}. If we consider $X=\nu_d(\PP V)\subset \PP U$ and $\Sym^{2d}V=V_{2d}$, we have the short exact sequence $$
		0\rightarrow \mathcal I_X\rightarrow \OO_{\PP U}\rightarrow \OO_X\rightarrow 0.
		$$
		Twisting it by $\OO_{\PP U}(2)$ and taking the long exact sequence of cohomologies we obtain $$
		0\rightarrow H^0(\mathcal I_X(2))\rightarrow H^0(\OO_{\PP U}(2))\rightarrow H^0(X,\OO_X(2))\rightarrow 0.$$ 
		Notice that the last map is a surjection since $H^0(X,\OO_X(2))=V_{2d}$ that is irreducible. 
		
		We remark the follow identifications: $H^0(\mathcal I_X(2))$ is given by the quadric forms on the ideal sheaf of $X$, that is, the quadric forms that belong to $ C$. $H^0(\OO_{\PP U}(2))=\Sym^2(U)=\Sym^2(\Sym^dV)$ and $H^0(X,\OO_X(2))=\Sym^{2d} V$.
		
		Assume for the purpose of contradiction that $\Sigma_1(\nu_2(\PP U))\cap C\neq \emptyset$ and $\Sigma_2(\nu_2(\PP U))\cap  C\neq \emptyset$, this means that there exists polynomials $f,g\in C$ of respective ranks $1$ and $2$ in $\PP\Sym^2 U$ such that $f,g\in H^0(\mathcal I_X(2))$. This implies that $X$ is contained in the hyperplane determined by $f=l^2$ and in the union of hyperplanes determined by $g=l_0^2+l_1^2=(l_0+il_1)(l_0-il_1)$, where $l,l_0,l_1$ are linear forms in $\PP U$. However $X=\nu_2(\PP U)$ is not contained in any {hyperplane}, thus the intersection must be empty.
	\end{proof}
	\begin{lemma}\label{lemma: inj}
		The map $\pi_C$ is injective in $\nu_2(\PP U)$.
	\end{lemma}
	\begin{proof}
	     Let $x,\ y$ be elements both in $\nu_2(\PP U)$. The map is given by $$
	x\mapsto \overline{x,C}\cap \Sym^{2d}V.	$$
	Thus, the equation $\pi_C(x)=\pi_C(y)$ implies that $\overline{x,C}=\overline{y,C}$. This means that there exists $\lambda\in \CC$ and $c\in C$ such that $x=\lambda y+c$. Therefore, $x-\lambda y=c\in C$ which is a contradiction since $x-\lambda y\in \Sigma_{2}(\nu_2(U))$.
	\end{proof}
	
	\begin{lemma}
The projection $\pi_C:\Sym^2(\Sym^d V)\rightarrow \Sym^{2d}V$ restricted to the second secant variety of the Veronese variety $\Sigma_2(\nu_2(\Sym^dV))$ is injective. 
\end{lemma}
\begin{proof}
    Consider the projection $$\Sym^dV\times \Sym^dV\times \Sym^2(\Sym^d V)\rightarrow \Sym^2(\Sym^dV).$$ Let $\text{Ab}_2(\nu_2(\Sym^dV))=\{(\alpha,\beta,g)|\alpha^2+\beta^2=g\}$  be the abstract Veronese variety that under the projection is mapped to $\Sigma_2(\nu_2(\Sym^dV))$. Notice that the fiber of this projection on a point $g$ is $\mathrm O(2)$ by Lemma \ref{lemma: inj}.
    
    We may consider a similar projection $$\Sym^dV\times \Sym^dV\times \Sym^{2d} V\rightarrow \Sym^{2d}V.$$ 
    We may define $X=\{(\alpha,\beta,f)|\alpha^2+\beta^2=f\}$ in the same fashion as before. Under this projection we have that $X$ is mapped to $\mathrm{SOS}_2$ and the fiber on a point $f$ is $\mathrm{SOS}_2(f)=\mathrm O(2)$ by Lemma \ref{thm:sos2f}.
    
    Notice that the map $\Sym^{2}(\Sym^{d}V)\rightarrow \Sym^{2d}V$ that corresponds to the change of coordinates $w_1=x_0^d,\dots,w_N=x_n^d$ is injective when restricted to $\Sigma_2(\nu_2(\Sym^dV))$ and so is the induced linear map from $\text{Ab}_2$ to $X$.
    
    Joining those maps into a diagram we obtain:
    $$
\begin{tikzcd}
                      &  & \text{Ab}_2 \arrow[rd, "\psi"] \arrow[lld, "\varphi"'] &                                              \\
X \arrow[rrd, "\xi"'] &  &                                                        & \Sigma_2(\nu_2(\Sym^dV)) \arrow[ld, "\zeta"] \\
                      &  & \mathrm{SOS}_2                                                  &                                             
\end{tikzcd}
    $$
From the previous remarks, $\varphi$ is an one-to-one map and the fibers of $\psi$ and $\xi$ are both equal to $\mathrm O(2)$. Since the diagram commutes, we also obtain that $\zeta$ is a one-to-one map. 
\end{proof}

	\degSOSk*

	\begin{proof}
		Since $\Sigma_k(\nu_2(\PP U))\cap C=\emptyset$ and $\pi_C|_{\Sigma_k(\nu_2(\PP U))}$ is injective for $k=1,2$, it follows $\deg(\mathrm{SOS}_j)=\deg(\Sigma_j(\nu_2(\PP U)))$. A classical result by Segre \cite{HT1984} states that for any $j\leq N$ \[
		\deg(\Sigma_j(\nu_2(\PP U)))=\prod_{i=0}^{N-1-j}\frac{\binom{N+i}{N-j-i}}{\binom{2i+1}{i}}. \qedhere \] 
	\end{proof}

	We notice that in the case of $n=2$ and $d=2$ Theorem \ref{intersection} is sharp in the sense that for the $3$-secant variety of $\nu_2(\PP U)$ the intersection with $ C$ is non-empty. Indeed, one can find by computation that the intersection of $\Sigma_1(\nu_2(\PP U))$ and $\Sigma_2(\nu_2(\PP U))$ with $ C$ are empty. Thus, $\deg(\mathrm{SOS}_1)=32$ and $\deg(\mathrm{SOS}_2)=126$ as expected. However, for $\mathrm{SOS}_3$ the intersection has codimension $3$ in $\PP^5=\PP U$. When the intersection is non-empty, the degree of $\Sigma_k(\nu_2(\PP U))$ is still an upper bound for the degree of $\mathrm{SOS}_k$.
    
%    We observe that the degree results of Theorem~\ref{thm:degSOSk} align with the results on the algebraic degree of semi-definite programming~\cite[Corollary~15]{nie2010algebraic}. Here, the authors consider the largest degree out of the minimal polynomials of the optimal solution coordinates to the convex optimisation problem dual to a certain semi-definite programming problem. Through the sum of squares and moment duality, the number of squares, $k$, becomes the rank of the optimal solution. Then, their degree result is valid providing that equation~\eqref{eq: NRS} is satisfied. 

% cases I found for k < 20, n < 1000 and d < 100
% (k,n,d)
% (1,2,2)
% (3,1,7)
% (3,2,3)
% (5,1,1)
% (5,1,10)
% (6,2,4)
% (7,2,1)
% (8,6,2)
% (9,3,1)
% (10,2,5)
% (11,4,1)
% (12,2,2)
% (13,5,1)
% (14,1,7)
% (15,2,6)
% (15,4,3)
% (15,6,1)
% (17,7,1)
% (18,1,10)
% (18,2,3)
% (19,8,1)